\documentclass[reqno, amsmath, amssymb, amsthm, latexsym, 12pt]{amsart}
\usepackage[margin=1in]{geometry}
\usepackage[colorlinks=true]{hyperref}
\hypersetup{urlcolor=blue, citecolor=blue,linkcolor=blue}
\usepackage{color}

\usepackage{amssymb}

\usepackage{graphicx}

\newtheorem{theorem}{Theorem}[section]
\newtheorem{lemma}[theorem]{Lemma}

\theoremstyle{definition}

\newtheorem{lem}[theorem]{Lemma}

\newtheorem{ex}[theorem]{Example}

\theoremstyle{remark}
\newtheorem{remark}[theorem]{Remark}

\numberwithin{equation}{section}

\newcommand{\diag}{\operatorname{diag}}

\newcommand{\D}{\operatorname{D}}

\renewcommand{\P}{\mathbb{P}}
 
\newcommand{\N}{\mathbb{N}} 
\newcommand{\R}{\mathbb{R}}

\newcommand{\T}{\operatorname{UT}}  

\newcommand{\mtx}[1]{ \begin{bmatrix} #1 \end{bmatrix}}

\begin{document}

\title{The Karcher Mean on $2\times 2$ Triangular Matrices with Positive Diagonal Entries}


\author{Huajun Huang}
\address{Department of Mathematics and Statistics, Auburn University, Auburn, AL 36849, USA}
\curraddr{}
\email{huanghu@auburn.edu}
\thanks{}


\date{}

\dedicatory{}

\commby{}

\begin{abstract} 
	Let $\T_n$ denote the group of 
	$n\times n$ real upper triangular matrices with positive diagonal entries. 
	We derive a closed-form analytic expression for the Karcher mean on $\T_2$
	and obtain a sharp convergence threshold for the projection mean iteration with fixed initial diagonal, together with a sufficient convergence condition for arbitrary initial points. 
	In addition, we obtain   explicit formulas for the log-Euclidean mean and 
	establish a Lie--Trotter type formula for the Karcher mean on $\T_2$. 
	For $\T_n$ we derive the diagonal and first superdiagonal entries of every solution of the Karcher equation and transfer nonconvergence examples from $\T_2$.
\end{abstract}

\maketitle

\section{Introduction}

The notion of the Karcher mean, also known as the Riemannian center of mass, was introduced by Karcher \cite{Karcher1977} in the setting of Riemannian manifolds. In the context of positive definite matrices, it coincides with the Fr\'echet mean under the affine-invariant Riemannian metric and has been extensively studied by Moakher \cite{Moakher2005}, Bhatia and Holbrook \cite{HolbrookBhatia2006}, and Lawson and Lim \cite{LawsonLim2011,LawsonLim2014}. More generally, the Karcher mean has been developed in a wide range of settings, including Hilbert spaces, $C^*$-algebras, and Lie groups \cite{Lawson2020,Lawson2024,Lawson2025}.

Let $\omega=(w_1,\ldots,w_m)$ be a probability vector and let $A_1,\ldots,A_m$ be elements of a space equipped with an exponential map. The (weighted) Karcher mean $X=G(\omega;A_1,\ldots,A_m)$ is defined, when it exists uniquely, as the solution of the Karcher equation
\begin{equation}\label{Karcher-eq-intro}
	\sum_{i=1}^m w_i \exp_X^{-1}(A_i)=0.
\end{equation}
In  the cone $\mathbb{P}$ of positive definite matrices, this equation is equivalent to
\begin{equation}\label{Karcher-eq-P}
	\sum_{i=1}^m w_i \log\!\bigl(X^{-1/2}A_iX^{-1/2}\bigr)=0,
\end{equation}
which admits a unique solution for all $A_i\in\mathbb{P}$.
Despite its fundamental importance, explicit formulas for the Karcher mean are known only in limited situations: for instance, the two-variable case reduces to the weighted geometric mean, and commuting families admit a multiplicative expression. In general, however, no closed-form formula is available when $m \ge 3$ and the matrices do not commute.

Consequently, numerical schemes such as the projection mean algorithm have been developed to approximate the Karcher mean. Starting from an initial point $X_0$, the iteration is given by
\begin{equation}\label{projection mean algorithm}
	\log\!\bigl(X_k^{-1}X_{k+1}\bigr)
	= \frac{1}{c}\sum_{i=1}^m w_i \log\!\bigl(X_k^{-1}A_i\bigr),
	\qquad k\ge 0,
\end{equation}
where $c>0$ is a relaxation denominator (the step size is $1/c$). On spaces with nonpositive curvature, this iteration for $c=1$ is known to converge  under mild conditions.

In this paper, the Karcher mean refers to the solution of the Lie-group logarithmic equation $\sum_iw_i\log(X^{-1}A_i)=0$. We investigate this mean and the projection mean algorithm on the Lie group $\T_n$ of real upper triangular matrices with positive diagonal entries. This group admits the semidirect product decomposition
\[
\T_n = \D_n \ltimes \mathrm{T}_n,
\]
where $\D_n$ is the group of positive diagonal matrices and $\mathrm{T}_n$ is the group of unipotent upper triangular matrices. Although the projection mean algorithm with $c=1$ is known to converge on both $\D_n$ and $\mathrm{T}_n$ (cf. \cite{Lawson2025}), numerical evidence indicates that its behavior on $\T_n$ is fundamentally different: convergence is no longer guaranteed for any   step size $c$.

The main contribution of this paper is to derive an explicit analytic expression for  the Karcher mean and to develop a detailed analysis of  the projection mean iteration on   $\T_2$, which  reduces to a one-dimensional dynamical system governing the off-diagonal entry. 
In Section 2, we show that the Karcher mean on $\T_2$ admits a closed-form analytic expression. Let $D = \diag(d_1,d_2)$ be the diagonal part of  $A_1^{w_1}\cdots A_m^{w_m}$  and set
\[
D^{-1}A_i = \mtx{a_{11}^{(i)} &a_{12}^{(i)} \\[1em] 0  &a_{22}^{(i)}},\quad 
i=1,\ldots,m.
\] 
Let $f(x,y) = \dfrac{\log(x/y)}{x-y}$ for $x,y>0$ and set $f(x,x)=\dfrac{1}{x}$ by continuity. 
Then
\[
G(\omega; A_1,\ldots, A_m)=D \, G(\omega; D^{-1}A_1,\ldots, D^{-1}A_m)
= D  \exp\!\left( 
\frac{1}{c_0}\sum_{i=1}^{m}w_i 
\log \left(D^{-1}A_i\right)
\right)
\]
where $c_0=\sum_{i=1}^m w_i\, f\!\bigl(a_{22}^{(i)},a_{11}^{(i)}\bigr)\,  a_{22}^{(i)}$  is a constant no less than $1$. 
Explicitly,
\[
G(\omega; A_1,\ldots, A_m)
= \begin{bmatrix}
	d_1 
	&
	\dfrac{d_1\sum_{i=1}^m w_i\, f\!\bigl(a_{22}^{(i)},a_{11}^{(i)}\bigr)\, a_{12}^{(i)}}
	{\sum_{i=1}^m w_i\, f\!\bigl(a_{22}^{(i)},a_{11}^{(i)}\bigr)\,  a_{22}^{(i)} }
	\\[1em]
	0 &
	d_2 
\end{bmatrix}.
\]
Besides, the constants $c_0$ and $c$ characterize the fixed-diagonal convergence behavior
of the projection mean iteration \eqref{projection mean algorithm}. 
When the initial diagonal is $D$, the error in the $(1,2)$-entry is multiplied by $1-c_0/c$ at each step. Thus the sharp convergence condition for a nonstationary initial point with this diagonal is $c>c_0/2$, and $c=c_0$ gives the mean in one step. For arbitrary initial points, convergence still holds when $c>c_0/2$, but exceptional convergent trajectories may occur below this threshold. With at least two positive weights, no universal parameter $c>0$ ensures convergence for all input matrices. 

As further consequences of the explicit formula, we derive closed-form expressions for the log-Euclidean mean and establish a Lie–Trotter type formula for the Karcher mean on $\T_2$. These results provide additional insight into the interplay between different matrix means in the triangular setting.

Finally, we extend our analysis to $\T_n$ for $n \ge 3$. Using a recursive structure based on principal submatrices, we show that several key features persist in higher dimensions. In particular, the absence of a universal convergence parameter $c$ holds for all $n$.

The results of this paper provide one of the few instances where the
Karcher mean admits a fully explicit analytic description in a genuinely
noncommutative setting, and they reveal new dynamical phenomena for iterative algorithms on $\T_n$ and other solvable Lie groups.
Besides, our methods and results extend naturally to the group of $n\times n$ complex upper triangular matrices with positive diagonal entries.


\section{The Karcher mean and the projection mean iteration on $\T_2$}

In this section, we give the explicit formula of the  Karcher mean on $\T_2$  and provide the convergence behavior of the projection mean iteration \eqref{projection mean algorithm} on $\T_2$ with respect to different factors $c$.


\begin{lemma}\label{thm: log 2x2 UT}
	Suppose that $a_{11}, a_{22}>0$. Then 
	\begin{equation}\label{2x2 upp tri log}
		\log\left(\begin{bmatrix}
			a_{11} &a_{12} \\ 0 &a_{22}
		\end{bmatrix}\right)=\begin{bmatrix}
			\log a_{11}  &\frac{\log(a_{22}/a_{11})}{a_{22}-a_{11}} a_{12}\\0 &\log a_{22} 
		\end{bmatrix}.
	\end{equation}
\end{lemma}

\begin{proof}
	Without loss of generality, we assume that $a_{11}\ge a_{22}>0.$ Then
	\[
	\begin{aligned}
		\log\left(\begin{bmatrix}
			a_{11} &a_{12} \\ 0 &a_{22}
		\end{bmatrix}\right)
		&= 
		(\log a_{11} )I_2+	\log\left(\begin{bmatrix}
			1 &a_{12}/a_{11} \\ 0 &a_{22}/a_{11}
		\end{bmatrix}\right)
		\\
		&= 
		(\log a_{11} )I_2 - \sum_{p=1}^{\infty} \frac{1}{p} 
		\begin{bmatrix}
			0 &-\frac{a_{12}}{a_{11}} \\ 0 &\frac{a_{11}-a_{22}}{a_{11}}
		\end{bmatrix}^{p}
		\\
		&= 
		(\log a_{11} )I_2 - \sum_{p=1}^{\infty} \frac{1}{p} 
		\begin{bmatrix}
			0 &-\frac{a_{12}}{a_{11}}\left(\frac{a_{11}-a_{22}}{a_{11}}\right)^{p-1} \\[1em] 0 &\left(\frac{a_{11}-a_{22}}{a_{11}}\right)^{p}
		\end{bmatrix}
		\\
		&= 
		\begin{bmatrix}
			\log a_{11}  &\frac{\log(a_{22}/a_{11})}{a_{22}-a_{11}} a_{12}
			\\[1em] 0 &\log a_{22} 
		\end{bmatrix}. \qquad  \qquad  \qedhere 
	\end{aligned}
	\]
\end{proof}

Lemma \ref{thm: log 2x2 UT}   reveals that the $(1,2)$-entry of the logarithm is governed by a divided difference of the logarithm. To analyze this dependence systematically, we introduce the following function: 
\begin{equation}\label{f function}
	f(x,y) := \frac{\log(x/y)}{x-y}=\frac{\log(x)-\log(y)}{x-y},\qquad x,y\in (0,\infty),
\end{equation}
and set $f(x,x) := \lim_{y\to x} f(x,y)=\dfrac{1}{x}$ by continuity. 

\begin{lemma}\label{thm: f fun prop}
	The function $f$ defined in \eqref{f function} satisfies the following properties. 
	\begin{enumerate}
		\item $f(x,y)=f(y,x)$  for all $x,y\in (0,\infty)$.
		\item $f(cx,cy)=\dfrac{1}{c}f(x,y)$  for all $c, x,y\in (0,\infty)$.
		\item Let $x_1,x_2,\dots,x_m>0$ satisfy  
		$x_1^{w_1}x_2^{w_2}\cdots x_m^{w_m}=1$ for a probability vector $(w_1,\ldots,w_m)$. Then 
		\begin{equation}\label{Jensen ineq}
			\sum_{i=1}^m w_i \frac{\log x_i}{x_i-1} =\sum_{i=1}^m w_i f(x_i,1) =\sum_{i=1}^m w_i f(1/x_i,1)  \ge 1.
		\end{equation}
		Moreover,   equality in  \eqref{Jensen ineq} holds if and only if $x_i=1$ for all $i$ with $w_i>0.$ 
	\end{enumerate}
\end{lemma}

\begin{proof}
	Properties (1) and (2) are straightforward. Let us prove (3) here.
	
	Let $x_1,\ldots,x_m>0$ satisfy 
	$
	x_1^{w_1}x_2^{w_2}\cdots x_m^{w_m}=1.
	$
	Then
	\[
	\sum_{i=1}^m w_i \log x_i
	= \log\!\bigl(x_1^{w_1}x_2^{w_2}\cdots x_m^{w_m}\bigr)
	= 0.
	\]
	Hence
	\[
	\begin{aligned}
		\sum_{i=1}^m w_i f(x_i,1)
		&= \sum_{i=1}^m w_i \frac{\log x_i}{x_i-1}
		= \sum_{i=1}^m w_i \left( \frac{1}{x_i-1}+1 \right)\log x_i \\
		&= \sum_{i=1}^m w_i \frac{x_i \log x_i}{x_i-1}
		= \sum_{i=1}^m w_i \frac{\log(1/x_i)}{1/x_i-1}
		= \sum_{i=1}^m w_i f(1/x_i,1).
	\end{aligned}
	\]
	
	Next, consider
	\[
	g(y):=f(e^y,1)=\frac{y}{e^y-1}, \qquad y\in\mathbb{R},
	\]
	with $g(0):=1$ by continuity. The function $g$ is analytic on $\mathbb{R}$. For $y\ne 0$,
	\[
	g'(y)=\frac{(e^y-1)-y e^y}{(e^y-1)^2}, \qquad
	g''(y)=\frac{e^y\bigl(y(e^y+1)-2e^y+2\bigr)}{(e^y-1)^3}.
	\]
	Let
	\[
	h(y)=y(e^y+1)-2e^y+2.
	\]
	Then $h(0)=0$ and for $y\ne 0$,
	\[
	h'(y)=y e^y + 1 - e^y >0.
	\]
	It follows that $h(y)<0$ for $y<0$ and $h(y)>0$ for $y>0$. 
	Hence, for all $y\ne 0$,
	\[
	g''(y)=\frac{e^y h(y)}{(e^y-1)^3}>0.
	\]
	Moreover, by L'H\^{o}pital's rule, $g''(0)=\tfrac{1}{6}>0$. Thus $g$ is strictly convex on $\mathbb{R}$.
	
	Applying Jensen's inequality to $g$, we obtain
	\[
	\begin{aligned}
		\sum_{i=1}^m w_i \frac{\log x_i}{x_i-1}
		&= \sum_{i=1}^m w_i f(x_i,1)
		= \sum_{i=1}^m w_i g(\log x_i) \\
		&\ge g\!\left(\sum_{i=1}^m w_i \log x_i\right)
		= g(0)=1.
	\end{aligned}
	\]
	This proves \eqref{Jensen ineq}. 
	
	Since $g$ is strictly convex, equality holds if and only if $\log x_i$ are all equal for indices with $w_i>0$, that is, $x_i=1$ for all $i$ with $w_i>0$.
\end{proof}

The properties of the function $f(x,y)$ established above will play a central role in the analysis of the Karcher mean and the projection mean iteration \eqref{projection mean algorithm}.    
Let 
\begin{itemize}
	\item 
	$\diag(A)$ denote the diagonal matrix formed from the diagonal entries of $A$,  and
	
	\item 
	$\diag(a_1,\cdots,a_n)$ denote the diagonal matrix with diagonal entries $a_1,\ldots,a_n$.
	
\end{itemize}
We will show that the convergence behavior of the sequence $\{X_k\}$ generated by the iteration \eqref{projection mean algorithm} on $\T_2$, once its diagonal is fixed, reduces to a one-dimensional dynamical system. Here is the reason for the normalization and the constant in the next theorem. Put $D=\diag(d_1,d_2)$, where $d_j=\prod_i(A_i)_{jj}^{w_i}$, and write
\[
D^{-1}A_i=\begin{bmatrix}a_{11}^{(i)}&a_{12}^{(i)}\\0&a_{22}^{(i)}\end{bmatrix},
\qquad \prod_i(a_{jj}^{(i)})^{w_i}=1\quad(j=1,2).
\]
Every solution of the Karcher equation has diagonal $D$. If $E_{12}$ denotes the $(1,2)$ matrix unit, every matrix with this diagonal is uniquely $X(t)=D(I+tE_{12})$. Consequently
\[
X(t)^{-1}X(s)=I+(s-t)E_{12},\qquad
\log(X(t)^{-1}X(s))=(s-t)E_{12}.
\]
The relative logarithms of successive fixed-diagonal iterates are therefore nilpotent and have just one free entry. By \eqref{2x2 upp tri log}, the $(1,2)$-entry of the Karcher residual at $X(t)$ is
\[
\sum_iw_i f(a_{22}^{(i)},a_{11}^{(i)})
\bigl(a_{12}^{(i)}-t a_{22}^{(i)}\bigr)=b-c_0t,
\]
where
\[
\begin{aligned}
	b&:=\sum_iw_i f(a_{22}^{(i)},a_{11}^{(i)})a_{12}^{(i)},\\
	c_0&:=\sum_iw_i f(a_{22}^{(i)},a_{11}^{(i)})a_{22}^{(i)}.
\end{aligned}
\]
Thus $-c_0$ is the slope of the scalar residual, and the iteration has multiplier $1-c_0/c$ on this fixed-diagonal set. Theorem~\ref{thm: 2x2 convergence 1} makes this reduction explicit; the zero $t=b/c_0$ gives the formula in Theorem~\ref{thm: 2x2 Karcher mean}.

\begin{theorem}\label{thm: 2x2 convergence 1}
	Let \(A_1,\ldots,A_m\in \T_2\). 
	Denote
	\begin{align}
		D &:= \diag(d_1,d_2)= \diag (A_1^{w_1}\cdots A_m^{w_m}),
		\\
		D^{-1}A_i &:= \mtx{a_{11}^{(i)} &a_{12}^{(i)} \\[0.5em] 0  &a_{22}^{(i)}},
		\quad i=1,\ldots,m,
		\\ \label{2x2 constant c}
		c_0
		&:=  	\sum_{i=1}^m
		w_i  f\left( a_{22}^{(i)}/a_{11}^{(i)}, \, 1\right ) 
		=
		\sum_{i=1}^m
		w_i \frac{a_{22}^{(i)}}{a_{22}^{(i)}-a_{11}^{(i)}}
		\log \left(\frac{a_{22}^{(i)}}{a_{11}^{(i)}}\right) .
	\end{align}	
	Suppose \(X_0\in \T_2\) satisfies $\diag (X_0)= D$.
	Then the sequence \(\{X_k \}\) generated by the iteration~\eqref{projection mean algorithm}  satisfies  
	\[\diag(X_k)=D,\qquad k= 0, 1, 2, \ldots \]
	and the \((1,2)\)-entries of $X_k^{-1}X_{k+1}$  (resp.  
	$\log \left(X_k^{-1}X_{k+1}\right)$) 
	form a geometric sequence with  ratio \(\left(1-\dfrac{c_0}{c}\right)\). 
\end{theorem}

\begin{proof}
	By \eqref{projection mean algorithm},
	\[
	\begin{aligned}
		\diag(X_{k+1})
		&=
		\diag(X_k)\exp\!\left(
		\frac{1}{c}\sum_{i=1}^{m}
		w_i\log \bigl(\diag(X_k)^{-1}\diag(A_i)\bigr)
		\right)
		\\
		&=
		\diag(X_k)^{1-\frac{1}{c}}
		\diag(A_1^{w_1}\cdots A_m^{w_m})^{\frac{1}{c}}
		=
		\diag(X_k)^{1-\frac{1}{c}}D^{\frac{1}{c}}.
	\end{aligned}
	\]
	By induction on $k$, we obtain $\diag(X_k)=D$ for all $k\in\mathbb{N}$. Hence each $X_k^{-1}X_{k+1}$ is unipotent upper triangular.
	
	Since $\diag(D^{-1}A_i)=\diag(a_{11}^{(i)},a_{22}^{(i)})$, we may write
	\begin{align} 
		X_k^{-1}A_i
		&=
		\begin{bmatrix}
			a_{11}^{(i)} & a_{12}^{(k,i)}\\
			0 & a_{22}^{(i)}
		\end{bmatrix}, \label{Xk_inv_Ai}
		\\ \label{Xk_inv_Xk+1}
		c\,\log(X_k^{-1}X_{k+1})
		&=
		\sum_{i=1}^m w_i \log\bigl(X_k^{-1}A_i\bigr)
		=
		\begin{bmatrix}
			0 & y_k\\
			0 & 0
		\end{bmatrix}. 
	\end{align}
	Thus
	\[
	X_k^{-1}X_{k+1}
	=
	\begin{bmatrix}
		1 & \dfrac{y_k}{c}\\
		0 & 1
	\end{bmatrix}.
	\]
	Moreover,
	\begin{align}
		\sum_{i=1}^m w_i \log \bigl(X_{k+1}^{-1}A_i\bigr)
		&=
		\sum_{i=1}^m
		w_i \log \!\left(
		(X_k^{-1}X_{k+1})^{-1}(X_k^{-1}A_i)
		\right)
		\nonumber\\
		&=
		\sum_{i=1}^m
		w_i \log \!\left(
		\begin{bmatrix}
			a_{11}^{(i)} &
			a_{12}^{(k,i)}-\dfrac{y_k}{c}a_{22}^{(i)}\\
			0 &
			a_{22}^{(i)}
		\end{bmatrix}
		\right).
		\label{X_k+1 A_i}
	\end{align}
	
	By \eqref{2x2 upp tri log}, \eqref{Xk_inv_Ai}, and \eqref{Xk_inv_Xk+1}, the $(1,2)$-entry of
	$\sum_{i=1}^m w_i \log (X_k^{-1}A_i)$ is
	\begin{equation}\label{y_k}
		y_k
		=
		\sum_{i=1}^m
		w_i \frac{a_{12}^{(k,i)}}{a_{22}^{(i)}-a_{11}^{(i)}}
		\log \!\left(\frac{a_{22}^{(i)}}{a_{11}^{(i)}}\right).
	\end{equation}
	Using \eqref{X_k+1 A_i}, \eqref{2x2 upp tri log}, \eqref{y_k}, and \eqref{2x2 constant c}, the $(1,2)$-entry of
	$\sum_{i=1}^m w_i \log (X_{k+1}^{-1}A_i)$ is
	\begin{align*}
		y_{k+1}
		&=
		\sum_{i=1}^m
		w_i \frac{a_{12}^{(k,i)}-\frac{y_k}{c}a_{22}^{(i)}}{a_{22}^{(i)}-a_{11}^{(i)}}
		\log \!\left(\frac{a_{22}^{(i)}}{a_{11}^{(i)}}\right)
		\\
		&=
		y_k
		-
		\frac{y_k}{c}
		\sum_{i=1}^m
		w_i \frac{a_{22}^{(i)}}{a_{22}^{(i)}-a_{11}^{(i)}}
		\log \!\left(\frac{a_{22}^{(i)}}{a_{11}^{(i)}}\right)
		\\
		&=
		\left(1-\frac{c_0}{c}\right)y_k.
	\end{align*}
	Therefore, the $(1,2)$-entries of $X_k^{-1}X_{k+1}$ (resp. $\log(X_k^{-1}X_{k+1})$) form a geometric sequence with ratio $1-\frac{c_0}{c}$.
	
	Let $x_i=a_{22}^{(i)}/a_{11}^{(i)}$ for $i=1,\ldots,m$. Then $x_i>0$ and
	$x_1^{w_1}\cdots x_m^{w_m}=1$. By Lemma~\ref{thm: f fun prop},
	\[
	c_0
	=
	\sum_{i=1}^m w_i f(x_i,1)
	\ge 1.
	\]
	This completes the proof.
\end{proof}

\begin{remark} The constant $c_0$ depends only on the weights and the normalized diagonal entries $a_{11}^{(i)},a_{22}^{(i)}$, and not on $a_{12}^{(i)}$. As shown above, it is the negative slope of the scalar residual $b-c_0t$, and $1-c_0/c$ is the multiplier of the fixed-diagonal iteration.
	Using the identities
	\[
	(a_{11}^{(1)})^{w_1}\cdots (a_{11}^{(m)})^{w_m}=1=(a_{22}^{(1)})^{w_1}\cdots (a_{22}^{(m)})^{w_m},
	\]
	we see that   $c_0$ admits the following equivalent expressions: 
	\[
	\begin{aligned}
		c_0 
		&=\sum_{i=1}^m w_i f\left (\frac{a_{11}^{(i)}}{a_{22}^{(i)}},1\right ) 
		=\sum_{i=1}^m w_i f(a_{22}^{(i)},a_{11}^{(i)})\, a_{22}^{(i)} 
		\\ 
		&=  \sum_{i=1}^m w_i f\left(\frac{a_{22}^{(i)}}{a_{11}^{(i)}},1\right )
		=
		\sum_{i=1}^m w_i f(a_{22}^{(i)},a_{11}^{(i)})\, a_{11}^{(i)}.
	\end{aligned}
	\]
	Moreover, \eqref{Jensen ineq} implies that $c_0\ge 1$. 
\end{remark}

Theorem \ref{thm: 2x2 convergence 1}   enables us to identify the equilibrium point explicitly, leading to a closed-form expression for the Karcher mean on $\T_2$ as follows.

\begin{theorem}\label{thm: 2x2 Karcher mean}
	The   Karcher mean $G(\omega; A_1,\ldots, A_m)$ on $\T_2$  admits the analytic expression
	\begin{equation}\label{2x2 Karcher mean}
		G(\omega; A_1,\ldots, A_m)
		= D \exp\!\left( 
		\frac{1}{c_0}\sum_{i=1}^{m}w_i 
		\log \left(D^{-1}A_i\right)
		\right)
	\end{equation}
	in which the $D$ in \eqref{2x2 Karcher mean} can be replaced by
	any $X_0\in\T_2$ with $\diag(X_0)=D$. 
	More explicitly, 
	\begin{equation}
		\label{2x2 Karcher mean 2}
		G(\omega; A_1,\ldots, A_m) = \begin{bmatrix}
			d_1 
			&
			\dfrac{d_1\sum_{i=1}^m w_i\, f\!\bigl(a_{22}^{(i)},a_{11}^{(i)}\bigr)\, a_{12}^{(i)}}
			{\sum_{i=1}^m w_i\, f\!\bigl(a_{22}^{(i)},a_{11}^{(i)}\bigr)\,  a_{22}^{(i)} }
			\\[1em]
			0 &
			d_2 
		\end{bmatrix}.
	\end{equation}
\end{theorem}

\begin{proof} By Theorem  \ref{thm: 2x2 convergence 1},
	when $c=c_0$ and $X_0=D$ in \eqref{projection mean algorithm}, the sequence $\{y_k \}$ in \eqref{y_k} is a geometric sequence with  ratio $0$, so that $y_1=y_2=\cdots =0$. Therefore, $X_1=X_2=\cdots$, and 
	the   Karcher mean has the analytic expression \eqref{2x2 Karcher mean}. Moreover, we can replace $D$ by any $X_0\in\T_2$ with the same diagonal as $D$ and get the   Karcher mean by one step of iteration  \eqref{projection mean algorithm}. 
	
	Every solution must have diagonal $D$, and on this diagonal the residual is $(b-c_0t)E_{12}$, as computed before Theorem~\ref{thm: 2x2 convergence 1}. Since $c_0\ge1$, $t=b/c_0$ is its unique zero. This also proves uniqueness and, because $f$ is analytic on $(0,\infty)^2$, analytic dependence on the input entries.
	
	Using \eqref{2x2 constant c}, 
	we evaluate the RHS of \eqref{2x2 Karcher mean} explicitly,
	\begin{eqnarray*}
		G(\omega; A_1,\ldots, A_m)
		&=& D\exp \!\left( 
		\frac{1}{c_0}
		\sum_{i=1}^{m}
		w_i \mtx{\log (a_{11}^{(i)}) &\frac{a_{12}^{(i)}\log\left(a_{22}^{(i)} / a_{11}^{(i)} \right) }{a_{22}^{(i)}-a_{11}^{(i)}}\\[1em]
			0 & \log (a_{22}^{(i)}) } \right)
		\\ &=& D\exp \!\left(  \frac{1}{c_0} \mtx{\sum_{i=1}^{m} w_i\log (a_{11}^{(i)}) &\sum_{i=1}^{m} w_i f(a_{22}^{(i)}, a_{11}^{(i)}) \, a_{12}^{(i)} \\[1em] 0 &\sum_{i=1}^{m} w_i \log (a_{22}^{(i)}) } \right)
		\\
		&=& D \mtx{ 
			1 &\frac{1}{c_0} \sum_{i=1}^{m}w_i f(a_{22}^{(i)}, a_{11}^{(i)}) \, a_{12}^{(i)}\\[1em]
			0 &1 }
		\\
		&=& 
		\begin{bmatrix}  \displaystyle 
			d_1   & \dfrac{ d_1 \sum_{i=1}^{m} w_i f(a_{22}^{(i)},a_{11}^{(i)})\, a_{12}^{(i)} }{\sum_{i=1}^{m} w_i f(a_{22}^{(i)},a_{11}^{(i)})\, a_{22}^{(i)} } 
			\\[1em]
			0 &  d_2 
		\end{bmatrix}.
	\end{eqnarray*}
	We obtain the explicit form \eqref{2x2 Karcher mean 2} of the Karcher mean. 
\end{proof}

We are now in a position to analyze the convergence behavior of the projection mean iteration. In particular,  Theorem \ref{thm: 2x2 convergence 1} allows us to determine precisely when the iteration converges to the Karcher mean.

\begin{theorem}\label{thm: 2x2 normalized convergence}
	In the projection mean iteration~\eqref{projection mean algorithm}, suppose $\operatorname{diag}(X_0)=D$ and $X_0 \neq G(\omega;A_1,\ldots,A_m)$. Then the sequence $\{X_k\}$  converges to the Karcher mean $G(\omega;A_1,\ldots,A_m)$ if and only if
	\[
	c > \frac{c_0}{2}.
	\]
	In particular, the following statements hold.
	\begin{enumerate}
		\item The iteration~\eqref{projection mean algorithm} diverges whenever $c \in (0,\,1/2]$.
		
		\item 
		If the weight vector has at least two positive entries, there exists no constant $c>0$ for which the iteration \eqref{projection mean algorithm}, initialized at $X_0=D$, converges for every choice of $A_1,\ldots,A_m\in\T_2$.
	\end{enumerate}
	
\end{theorem}

\begin{proof}
	Suppose    $X_0\ne G(\omega; A_1,\ldots, A_m)$. Then $X_0\ne X_1$ and by
	\eqref{Xk_inv_Xk+1}, $y_0\ne 0$. 
	Thus, the sequence $\{X_k \}$ converges if and only if \(|1-\frac{c_0}{c}|<1\),  that is, if and only if \(c>c_0/2\). 
	\begin{enumerate}
		\item 
		Since $c_0\ge 1$, it follows that $c \le 1/2$ implies $c \le c_0/2$. Hence the iteration \eqref{projection mean algorithm}  diverges for   $c\in (0,1/2]$. 
		
		\item Relabel the indices so that $w_1,w_2>0$. For $s>0$ set
		\[
		r_1=e^{-s},\quad r_2=e^{w_1s/w_2},\quad r_i=1\ (i\ge3),
		\]
		and choose
		\[
		A_1=\begin{bmatrix}r_1&1\\0&1\end{bmatrix},\qquad
		A_i=\begin{bmatrix}r_i&0\\0&1\end{bmatrix}\quad(i\ge2).
		\]
		Then $D=I_2$, $b=w_1f(r_1,1)>0$, and
		$c_0=\sum_iw_if(r_i,1)\ge w_1s/(1-e^{-s})\to\infty$.
		For any prescribed $c>0$, choose $s$ with $c_0>2c$.
		The initial point $X_0=D$ is not the mean, so the iteration does not converge.
		\qedhere 
		
	\end{enumerate}
\end{proof}

Theorem \ref{thm: 2x2 normalized convergence} provides a sharp threshold for convergence in terms of the parameter $c$ when $X_0$ has the same diagonal as the Karcher mean. 
We now extend the analysis to arbitrary initial conditions.

When a sequence of matrices $\{Y_k\}$ converges to $Y$,   the \emph{asymptotic convergence rate} of  $\{Y_k\}$ is defined as 
\begin{equation}
	\limsup_{k\to\infty} \|Y_{k}-Y \|_{1}^{1/k}
\end{equation}
where $\|\cdot\|_1$ denotes the 1-norm.


\begin{theorem}\label{thm: gen conv behavior}
	Let $A_1,\ldots,A_m,X_0\in\T_2$ and $c>0$. Then:
	\begin{enumerate}
		\item If $c>1/2$, the diagonal of $X_k$ converges to $D$; the diagonal entries of $\log(X_k^{-1}X_{k+1})$ are geometric with ratio $1-1/c$.
		\item If $c>c_0/2$, then $X_k\to G(\omega;A_1,\ldots,A_m)$, and the asymptotic convergence rates of both $X_k-G$ and $\log(X_k^{-1}X_{k+1})$ are at most
		\begin{equation}\label{bound_asy_conv_rate}
			\max\{|1-1/c|,\,|1-c_0/c|\}.
		\end{equation}
		\item If $0<c\le1/2$ and $X_0\ne G(\omega;A_1,\ldots,A_m)$, the sequence $\{X_k\}$ does not converge.
	\end{enumerate}
\end{theorem}

\begin{proof}
	Write $D^{-1}X_k=\begin{bmatrix}p_k&p_kt_k\\0&q_k\end{bmatrix}$ with $p_k,q_k>0$. Equating diagonal entries in the iteration gives, with $r=1-1/c$,
	\begin{equation}\label{diag_Xk}
		p_k=p_0^{r^k},\qquad q_k=q_0^{r^k}.
	\end{equation}
	This also proves the geometric assertion in (1).
	
	To handle the off-diagonal entry without assuming a fixed diagonal, put
	\[
	H(p,q)=\sum_iw_i f(a_{11}^{(i)}/p,a_{22}^{(i)}/q)\frac{a_{12}^{(i)}}p,
	\qquad
	K(p,q)=\sum_iw_i f(a_{11}^{(i)}/p,a_{22}^{(i)}/q)\frac{a_{22}^{(i)}}q.
	\]
	The weighted logarithmic residual at $X_k$ has entries
	\[
	\sum_iw_i\log(X_k^{-1}A_i)
	=\begin{bmatrix}-\log p_k&H(p_k,q_k)-K(p_k,q_k)t_k\\0&-\log q_k\end{bmatrix}.
	\]
	Define the continuous exponential divided difference
	\[
	\psi(u,v)=\begin{cases}(e^u-e^v)/(u-v),&u\ne v,\\ e^u,&u=v,\end{cases}
	\quad
	h_k=\frac{p_k^{1/c}}c\psi(-\log p_k/c,-\log q_k/c).
	\]
	Direct exponentiation in \eqref{projection mean algorithm} yields the exact affine recurrence
	\begin{equation}\label{affine_tk}
		t_{k+1}=\mu_kt_k+\nu_k,\qquad
		\mu_k=(p_k/q_k)^{1/c}-h_kK(p_k,q_k),\quad
		\nu_k=h_kH(p_k,q_k).
	\end{equation}
	In particular, $H(1,1)=b$, $K(1,1)=c_0$, and the mean has coordinate $t_*=b/c_0$.
	
	Suppose $c>c_0/2$. As $c_0\ge1$, $\rho=|1-1/c|<1$ and $\lambda=1-c_0/c$ satisfies $|\lambda|<1$. If $\rho>0$, \eqref{diag_Xk} and smoothness near $(1,1)$ give
	\[
	\mu_k=\lambda+O(\rho^k),\qquad
	\nu_k=b/c+O(\rho^k).
	\]
	For $e_k=t_k-t_*$, equation \eqref{affine_tk} therefore becomes
	$e_{k+1}=\mu_ke_k+O(\rho^k)$. Given any $a$ with
	$\max\{\rho,|\lambda|\}<a<1$, eventually $|\mu_k|<a$; iteration of this inequality shows $e_k=O(a^k)$ (enlarging $a$ slightly first if necessary). This proves convergence and the claimed bound for $X_k-G$ by letting $a$ decrease to the displayed maximum. If $\rho=0$, then $c=1$, and $p_k=q_k=1$ for $k\ge1$; the fixed-diagonal recurrence gives the same conclusion. The residual is analytic and vanishes at $G$, so
	$\log(X_k^{-1}X_{k+1})=c^{-1}\sum_iw_i\log(X_k^{-1}A_i)$ has the same rate bound.
	
	If $c\le1/2$ and $(p_0,q_0)\ne(1,1)$, \eqref{diag_Xk} yields an unbounded diagonal subsequence when $c<1/2$, or a nonconstant two-cycle when $c=1/2$. Otherwise the diagonal is $D$, and Theorem~\ref{thm: 2x2 normalized convergence} applies unless $X_0=G$.
\end{proof}

We demonstrate the main results of this section in the following   example.

\begin{ex}
	Suppose $m=3$, $\omega=(1/3,1/3, 1/3)$, and 
	\[
	A_1=\begin{bmatrix}
		1000000 & 1\\
		0 & 1
	\end{bmatrix},\quad
	A_2=\begin{bmatrix}
		0.001 & 2\\
		0 & 1
	\end{bmatrix},\quad
	A_3=\begin{bmatrix}
		0.001 & 3\\
		0 & 1
	\end{bmatrix}.
	\]
	Then by   \eqref{2x2 constant c}, 
	\[
	\begin{aligned}
		D &= \diag(A_1^{1/3}A_2^{1/3}A_3^{1/3})=I_2,
		\\
		c_0 &= \frac{1}{3}\frac{\log(10^6)}{10^6-1}+\frac{1}{3}\frac{\log(0.001)}{0.001-1}+\frac{1}{3}\frac{\log(0.001)}{0.001-1}
		\approx 4.6098. 
	\end{aligned}
	\]

	\begin{enumerate}
		\item If we perform the projection mean iteration   \eqref{projection mean algorithm} for 
		\(c=c_0\approx 4.6098\) and  $X_0\in\T_2$ such that $\diag(X_0)=I_2$, for example,
		$X_0=\begin{bmatrix}
			1 & -3\\
			0 & 1
		\end{bmatrix}$, 
		then  
		\[
		X_{1}=X_{2}=X_3=\cdots=G(\omega;A_1,A_2,A_3)
		\approx\begin{bmatrix}1&2.5\\0&1\end{bmatrix} 
		\]
		which is  the Karcher mean $G(\omega;A_1,A_2,A_3).$ 
		
		\item
		If we perform the iteration \eqref{projection mean algorithm} for
		\(c=2.4>c_0/2\) and an \(X_0\in\T_2\), then the sequence $\{X_k\}$ converges to $G(\omega;A_1,A_2,A_3).$
		For example, let   $X_0=I_2=D$, then a numerical computation in R  shows that
		\[
		\begin{aligned}
			X_1 &=
			\begin{bmatrix}
				1 & 4.8019\\
				0 & 1
			\end{bmatrix},
			&
			X_0^{-1}X_1 &=
			\begin{bmatrix}
				1 & 4.8019\\
				0 & 1
			\end{bmatrix},
			\\[6pt]
			X_2 &=
			\begin{bmatrix}
				1 & 0.3806\\
				0 & 1
			\end{bmatrix},
			& X_1^{-1}X_2
			&=
			\begin{bmatrix}
				1 & -4.4213\\
				0 & 1
			\end{bmatrix},
			\\[6pt]
			X_3 &=
			\begin{bmatrix}
				1 & 4.4514\\
				0 & 1
			\end{bmatrix},
			& X_2^{-1}X_3
			&=
			\begin{bmatrix}
				1 & 4.0709\\
				0 & 1
			\end{bmatrix},
			\\[6pt]
			X_4 &=
			\begin{bmatrix}
				1 & 0.7032\\
				0 & 1
			\end{bmatrix},
			&
			X_3^{-1}X_4 &=
			\begin{bmatrix}
				1 & -3.7482\\
				0 & 1
			\end{bmatrix},
			\\  
			&\vdots & &\vdots 
		\end{aligned}
		\]
		We see that the $(1,2)$-entries of $\{X_k^{-1}X_{k+1}\}$  (resp. $\{\log(X_k^{-1}X_{k+1})\}$) form a convergent geometric sequence with ratio $1-\frac{c_0}{c}\approx -0.9208.$

		\item
		If we perform the iteration \eqref{projection mean algorithm} for
		\(c=2.3<c_0/2\) and certain \(X_0\in\T_2 \), then the sequence diverges.
		For example, let  $X_0=I_2=D$, then a numerical computation in R shows that
		\[
		\begin{aligned}
			X_{1} &=
			\begin{bmatrix}
				1 & 5.0106\\
				0 & 1
			\end{bmatrix},
			&
			X_{0}^{-1}X_{1} &=
			\begin{bmatrix}
				1 & 5.0106\\
				0 & 1 
			\end{bmatrix},
			\\[6pt]
			X_{2} &=
			\begin{bmatrix}
				1 & -0.0213\\
				0 & 1
			\end{bmatrix},
			&
			X_{1}^{-1}X_{2} &=
			\begin{bmatrix}
				1 & -5.0319\\
				0 & 1 
			\end{bmatrix},
			\\[6pt]
			X_{3} &=
			\begin{bmatrix}
				1 & 5.032\\
				0 & 1
			\end{bmatrix},
			&
			X_{2}^{-1}X_{3} &=
			\begin{bmatrix}
				1 & 5.0534\\
				0 & 1 
			\end{bmatrix},
			\\[6pt]
			X_{4} &=
			\begin{bmatrix}
				1 & -0.0428\\
				0 & 1
			\end{bmatrix},
			&
			X_{3}^{-1}X_{4} &=
			\begin{bmatrix}
				1 & -5.0749\\
				0 & 1 
			\end{bmatrix},
			\\ &\vdots & &\vdots 
		\end{aligned}
		\]
		The $(1,2)$-entries of $\{X_k^{-1}X_{k+1}\}$ (resp. $\{\log(X_k^{-1}X_{k+1})\}$) form a divergent geometric sequence with ratio $1-\frac{c_0}{c}\approx -1.0043.$ 
		
	\end{enumerate}
\end{ex}

\section{The Log--Euclidean mean and the Lie--Trotter formula  on $\T_2$}

The explicit form \eqref{2x2 Karcher mean 2} of
the Karcher mean on $\T_2$ can be used to derive the corresponding Lie--Trotter formula. Let $E_{12}$ be the $2\times 2$ matrix with $1$ in the  $(1,2)$-entry and $0$ elsewhere. 

\begin{lem}
	For $A=\mtx{a_{ij}}\in\T_2$, $z\in\R$, and $p\in\R$,
	\begin{equation}\label{2x2 upp tri power} 
		(A+z E_{12})^{p} = A^{p}+  z  \, \frac{pf(a_{11}, a_{22})}{f(a_{11}^{p}, a_{22}^{p})} \,  E_{12} =
		A^{p}+  z \,  \frac{a_{11}^{p} - a_{22}^{p}}{a_{11} - a_{22}} \,  E_{12}. 
	\end{equation}
\end{lem}

\begin{proof} By \eqref{2x2 upp tri log}, we have
	\[
	\log(A+zE_{12}) = \log(A) + z f(a_{11},a_{22}) E_{12}.
	\]
	\eqref{2x2 upp tri log} also implies that for any $2\times 2$ upper triangular matrix $B=\mtx{b_{ij}}$ and $z\in\R$, 
	\begin{equation} \label{2x2 upp tri exp}
		\exp(B+z E_{12}) =\exp(B)+ \frac{z}{f(e^{b_{11}}, e^{b_{22}}) } E_{12}.
	\end{equation}
	Therefore, for $p\in\R$,
	\[
	\begin{aligned}
		(A+z E_{12})^{p}
		&= \exp\left(p\log(A+z E_{12}) \right)
		=  \exp\left(p\log(A) + p z f(a_{11},a_{22}) E_{12} \right)
		\\
		&= \exp (p\log(A))+ \frac{pz f(a_{11},a_{22})}{f(a_{11}^{p}, a_{22}^{p}) }  E_{12}
		=
		A^{p}+  z \frac{a_{11}^{p} - a_{22}^{p}}{a_{11} - a_{22}} E_{12}.
		\qquad \qedhere 
	\end{aligned}\]
\end{proof}

Given a probability vector $\omega=(w_1,\ldots,w_m)$ and matrices $A_1,\ldots,A_m\in\T_2$, the (weighted) \emph{log--Euclidean mean} 
of $A_1,\ldots,A_m$ is defined as
\begin{equation}\label{LE mean}
	LE(\omega; A_1,\ldots,A_m) 
	= \exp\left(\sum_{i=1}^{m} w_i\log(A_i) \right). 
\end{equation}
Denote 
\begin{equation}\label{A_i}
	A_i
	=\mtx{\widetilde{a}_{11}^{(i)} &\widetilde{a}_{12}^{(i)}\\ 0 &\widetilde{a}_{22}^{(i)}} = \diag(A_i) +\widetilde{a}_{12}^{(i)} \,  E_{12}. 
\end{equation}
The explicit form of log--Euclidean mean is described as follows. 

\begin{theorem}\label{thm: LE mean}
	On $\T_2$, the  log--Euclidean mean  
	of $A_1,\ldots,A_m\in\T_2$ is given by
	\begin{equation}\label{LE mean 2}
		LE(\omega; A_1,\ldots,A_m) 
		\, = \, 
		D + \frac{\sum_{i=1}^{m} w_i f(\widetilde{a}_{11}^{(i)}, \widetilde{a}_{22}^{(i)}) \,  \widetilde{a}_{12}^{(i)} }{f(d_1,d_2)} \,  E_{12}.
	\end{equation}
\end{theorem}

\begin{proof} By \eqref{2x2 upp tri log} and \eqref{2x2 upp tri exp},
	\begin{eqnarray*}
		LE(\omega; A_1,\ldots,A_m)
		&=& \exp \left(\sum_{i=1}^{m} w_i \log(A_i) \right)
		\\
		&=&  \exp \left(\sum_{i=1}^{m} w_i \left(\log(\diag(A_i) )+f(\widetilde{a}_{11}^{(i)}, \widetilde{a}_{22}^{(i)}) \, \widetilde{a}_{12}^{(i)}  \,  E_{12}   \right)\right)
		\\
		&=& \exp \left(\log(D)+ \left(\sum_{i=1}^{m} w_if(\widetilde{a}_{11}^{(i)}, \widetilde{a}_{22}^{(i)})  \, \widetilde{a}_{12}^{(i)}\right)  \,  E_{12}  \right)
		\\
		&=&
		D + \frac{\sum_{i=1}^{m} w_i f(\widetilde{a}_{11}^{(i)}, \widetilde{a}_{22}^{(i)}) \,  \widetilde{a}_{12}^{(i)} }{f(d_1,d_2)}  \,  E_{12}. 
	\end{eqnarray*}
	We get \eqref{LE mean 2}. 
\end{proof}

The Lie--Trotter  formula \cite{Trotter1959} on a Lie group $G$ can be expressed as:
\[
\lim_{p\to 0} (A_1^{w_1 p}\cdots A_m^{w_m p})^{1/p} = LE(\omega; A_1,\ldots,A_m),
\quad A_1,\ldots, A_m\in G. 
\]
The formula can be derived from the Baker–Campbell–Hausdorff formula. It has been extended to products and means on the cone $\P$ of positive definite matrices (cf. \cite{Kato1978, Araki1990, ANDO1994, AHN2007, BHATIA2012, HIAI2012, BHATIA2019B}.)  Here we prove the Lie--Trotter  formula for the Karcher mean on $\T_2$. 

\begin{theorem}[The Lie--Trotter Formula] \label{thm: Lie-Trotter}
	For a probability vector $\omega=(w_1,\ldots,w_m)$ and matrices $A_1,\ldots,A_m\in\T_2$, we have
	\begin{equation}
		\lim_{p\to 0} G(\omega; A_1^{p},\ldots,A_m^{p})^{1/p} 
		= LE(\omega; A_1,\ldots,A_m). 
	\end{equation}
\end{theorem}

\begin{proof}
	By \eqref{2x2 upp tri power},
	\begin{eqnarray*}
		A_i^{p} 
		&=& \diag(A_i) ^{p} + 
		\frac{(\widetilde{a}_{11}^{(i)})^{p} - (\widetilde{a}_{22}^{(i)})^{p}}
		{\widetilde{a}_{11}^{(i)} - \widetilde{a}_{22}^{(i)}}
		\, \widetilde{a}_{12}^{(i)}\, E_{12},
		\\
		D^{-p} A_i^{p}
		&=& D^{-p} \diag(A_i) ^{p}
		+ d_1^{-p} \widetilde{a}_{12}^{(i)}
		\frac{(\widetilde{a}_{11}^{(i)})^{p} - (\widetilde{a}_{22}^{(i)})^{p}}
		{\widetilde{a}_{11}^{(i)} - \widetilde{a}_{22}^{(i)}}\, E_{12}.
	\end{eqnarray*}
	Using \eqref{2x2 Karcher mean 2} and  
	$D^{-1}A_i= \diag(a_{11}^{(i)}, a_{22}^{(i)}) + a_{12}^{(i)} \, E_{12}$,
	we get
	\begin{eqnarray*}
		G(\omega;A_1^{p},\ldots,A_m^{p})
		&=&
		D^{p}
		+
		\frac{
			\sum_{i=1}^m w_i\, 
			f\!\left( (a_{22}^{(i)})^{p}, (a_{11}^{(i)})^{p}\right)\,
			\frac{(\widetilde{a}_{11}^{(i)})^{p} - (\widetilde{a}_{22}^{(i)})^{p}}
			{\widetilde{a}_{11}^{(i)} - \widetilde{a}_{22}^{(i)}}
			\, \widetilde{a}_{12}^{(i)}
		}{
			\sum_{i=1}^m w_i\,
			f\!\left( (a_{22}^{(i)})^{p}, (a_{11}^{(i)})^{p}\right)\,
			(a_{22}^{(i)})^{p}
		}
		\,E_{12},
		\\
		G(\omega;A_1^{p},\ldots,A_m^{p})^{1/p}
		&=& 
		D+ \frac{d_1-d_2}{d_1^{p}-d_2^{p}} \frac{
			\sum_{i=1}^m w_i\, 
			f\!\left( (a_{22}^{(i)})^{p}, (a_{11}^{(i)})^{p}\right)\,
			\frac{(\widetilde{a}_{11}^{(i)})^{p} - (\widetilde{a}_{22}^{(i)})^{p}}
			{\widetilde{a}_{11}^{(i)} - \widetilde{a}_{22}^{(i)}}
			\, \widetilde{a}_{12}^{(i)}
		}{
			\sum_{i=1}^m w_i\,
			f\!\left( (a_{22}^{(i)})^{p}, (a_{11}^{(i)})^{p}\right)\,
			(a_{22}^{(i)})^{p}
		}
		\,E_{12}.
	\end{eqnarray*}
	When $p\to 0$, we have 
	\[
	(a_{22}^{(i)})^{p}\to 1,\quad
	f\!\left( (a_{22}^{(i)})^{p}, (a_{11}^{(i)})^{p}\right)\to 1,
	\quad  
	\frac{x^p-y^p}{p}\to \log(x/y)
	\text{ for } x,y>0.
	\] 
	Therefore,
	\[
	\begin{aligned}
		\lim_{p\to 0} G(\omega;A_1^{p},\ldots,A_m^{p})^{1/p}
		& = D + 
		\frac{d_1-d_2}{\log(d_1/d_2)} \frac{
			\sum_{i=1}^m w_i\, 
			\frac{\log(\widetilde{a}_{11}^{(i)}/\widetilde{a}_{22}^{(i)})}
			{\widetilde{a}_{11}^{(i)} - \widetilde{a}_{22}^{(i)}}
			\, \widetilde{a}_{12}^{(i)}
		}{
			\sum_{i=1}^m w_i
		}
		\,E_{12}
		\\
		&= D + 
		\frac{
			\sum_{i=1}^m w_i\, 
			f(\widetilde{a}_{11}^{(i)},\widetilde{a}_{22}^{(i)})
			\, \widetilde{a}_{12}^{(i)}
		}{
			f(d_1,d_2)
		}
		\,E_{12}
		\\
		&= 
		LE(\omega; A_1,\ldots,A_m).
	\end{aligned}
	\]
	This completes the proof. 
\end{proof}

Theorem \ref{thm: Lie-Trotter} shows that the Karcher mean on $\T_2$ is compatible with the log-Euclidean mean under infinitesimal scaling.

\section{The Karcher mean on $\T_n$}

While the explicit formula obtained for $\T_2$ does not extend directly
to higher dimensions, the reduction to principal submatrices reveals that
the essential dynamical features persist on $\T_n$ for $n\ge 3$. In particular, the convergence behavior 
of the projection mean iteration remains a fundamentally local phenomenon
arising from $2 \times 2$ interactions.

For $1 \le p \le q \le n$, let $A[p:q]$ denote the principal submatrix of $A$ formed by the rows and columns indexed by $\{p,p+1,\ldots,q\}$.

\begin{lemma}\label{thm: submatrix iteration}
	Let $A_1,\ldots, A_m, X_0\in\T_n$.
	Let $\{X_k\}$ be the sequence generated by  the projection mean iteration
	\eqref{projection mean algorithm}. 
	\begin{enumerate}
		\item For any $p, q\in\{1,\ldots,n\}$ with $p\le q$, the iteration \eqref{projection mean algorithm} applied to $A_1[p:q],\ldots, A_m[p:q]$ with   initial matrix 
		$X_0[p:q]$ generates the sequence $\{X_k[p:q]\}$.
		\item If  $1\le p\le q\le n$ and 
		the iteration \eqref{projection mean algorithm} applied to $A_1[p:q],\ldots, A_m[p:q]$ with   initial matrix 
		$X_0[p:q]$ diverges, then  the sequence $\{X_k\}$ diverges.
	\end{enumerate}
	
\end{lemma}

\begin{proof}
	Let $p, q\in\{1,\ldots,n\}$ with $p\le q$. For each $k\in\N$, we have
	\[\begin{aligned}
		\log (X_{k}[p:q]^{-1}X_{k+1}[p:q])
		&=	\log (X_{k}^{-1}X_{k+1})[p:q]
		\\ &= 
		\left(\frac{1}{c}\sum_{i=1}^{m} w_i
		\log \left(X_k^{-1}A_i\right)\right)[p:q]
		\\ &= 
		\frac{1}{c}\sum_{i=1}^{m} w_i
		\log \left(X_k[p:q]^{-1}A_i[p:q]\right). 
	\end{aligned}
	\]
	This shows that the sequence $\{X_k[p:q]\}$ is generated by the same iteration applied to 
	$A_1[p:q],\ldots,A_m[p:q]$ with initial matrix $X_0[p:q]$, proving (1).
	
	Assertion (2) follows immediately from (1).
\end{proof}

Theorems \ref{thm: 2x2 Karcher mean}, \ref{thm: 2x2 normalized convergence}, \ref{thm: gen conv behavior}, and Lemma \ref{thm: submatrix iteration} imply the following   results about  the Karcher mean  and the projection mean iteration  on $\T_n$. 

\begin{theorem}
	Let  $n\ge 2$, $m\ge 2$, and $A_1,\ldots, A_m\in\T_n$. 
	Then every solution $X=\mtx{x_{p,q}}$ of the Karcher equation satisfies 
	\[
	\diag(X)  =D:= \diag (A_1^{w_1}\cdots A_m^{w_m}).
	\]
	Furthermore, write
	\[
	D=\diag(d_1,\ldots,d_n),\qquad
	D^{-1}A_i = \mtx{a_{p,q}^{(i)}},
	\qquad i=1,\ldots,m, 
	\]
	then
	\[
	x_{p,p+1} = 
	\frac{d_p \sum_{i=1}^m w_i\, f\!\bigl(a_{p+1,p+1}^{(i)},a_{p,p}^{(i)}\bigr)\, a_{p,p+1}^{(i)}}
	{\sum_{i=1}^m w_i\, f\!\bigl(a_{p+1,p+1}^{(i)},a_{p,p}^{(i)}\bigr)\,  a_{p+1,p+1}^{(i)} },\quad p=1,\ldots,n-1.
	\]
\end{theorem}

\begin{proof}
	The matrix $X$ satisfies the Karcher equation
	\[
	\sum_{i=1}^{m} w_i \log \left( X^{-1} A_i  \right) = 0.
	\]
	For $p,q\in\{1,\ldots,n\}$ with $p\le q$, we have
	\[
	\sum_{i=1}^{m} w_i \log \left( X[p:q]^{-1} A_i[p:q]  \right) = 
	\left( \sum_{i=1}^{m} w_i \log \left( X^{-1} A_i  \right)\right) [p:q]=0.
	\]
	Hence 
	\begin{equation}
		X[p:q] \text{ solves the Karcher equation for } A_1[p:q],\ldots,A_m[p:q],\qquad 1\le p\le q\le n.
	\end{equation}
	When $q=p\in\{1,\ldots,n\}$, we get
	\[
	x_{p,p}= (A_1)_{pp}^{w_1}\cdots (A_m)_{pp}^{w_m},
	\] 
	so that $\diag(X)=\diag (A_1^{w_1}\cdots A_m^{w_m})$. When
	$q=p+1\in\{2,\ldots,n\}$,
	Theorem \ref{thm: 2x2 Karcher mean} implies the expression of
	$x_{p,p+1}$. 
\end{proof}

\begin{theorem} Let  $n\ge 2$ and $m\ge 2$. 
	For any probability vector
	$\omega=(w_1,\ldots,w_m)$ with at least two positive weights, 
	there exists no constant \(c>0\) for which the sequence $\{X_k\}$ from the projection mean iteration
	\eqref{projection mean algorithm} converges for every choice of
	\(A_1,\ldots,A_m\in\T_n\).
\end{theorem}

\begin{proof} 
	By Theorem \ref{thm: 2x2 normalized convergence} (2), for any $c>0$ we can choose data in $\T_2$ for which the iteration diverges. Embed these matrices and the initial point as upper-left $2\times2$ blocks in $\T_n$, with identity matrices in the complementary diagonal block. Choose the full initial point so that its upper-left block is the $2\times2$ initial point.
	The projection mean iteration \eqref{projection mean algorithm} on the submatrices
	$$A_1[1:2],\ \ldots,\ A_m[1:2],\ X_0[1:2]\in\T_2$$ produces a divergent
	sequence, namely $\{X_k[1:2]\}$ by Lemma \ref{thm: submatrix iteration} (2). Hence, the sequence $\{X_k\}$  diverges. 
\end{proof}

\begin{theorem}
	Let $A_1,\ldots,A_m,X_0\in\T_n$, and put
	$D=\diag(A_1^{w_1}\cdots A_m^{w_m})$. For each $p=1,\ldots,n-1$, denote by $G_p$ the Karcher mean of $A_1[p:p+1],\ldots,A_m[p:p+1]$ and let $c_p$ be the corresponding constant $c_0$. The iteration \eqref{projection mean algorithm} does not converge if either
	\begin{enumerate}
		\item $0<c\le1/2$ and $\diag(X_0)\ne D$; or
		\item for some $p$, $\diag(X_0[p:p+1])=D[p:p+1]$, $X_0[p:p+1]\ne G_p$, and $0<c\le c_p/2$.
	\end{enumerate}
\end{theorem}
\begin{proof}
	The diagonal recurrence \eqref{diag_Xk} applied entrywise proves (1). In (2), Theorem~\ref{thm: 2x2 normalized convergence} shows that the indicated $2\times2$ iteration does not converge. Lemma~\ref{thm: submatrix iteration} then gives the conclusion.
\end{proof}

\bibliographystyle{amsplain}

\bibliography{Karchermean}

\end{document}